\documentclass[preprint]{elsarticle}
\usepackage{epic}
\usepackage{eepic}
\newsavebox{\wwide}
\newcommand{\wwidehat}[1]{\sbox{\wwide}{$#1$}
\ifdim\wd\wwide < 1.1 em \widehat{#1} \else
\setlength
{\unitlength}{0.01\wd\wwide}\overset
{\begin{picture}(100,6)
\path(0,0)(50,6)(100,0)
\end{picture}}{#1}\fi}

\ifx\pdfoutput\@undefined\usepackage[usenames,dvips]{color}
\else\usepackage[usenames,dvipsnames]{color}
\IfFileExists{pdfcolmk.sty}{\usepackage{pdfcolmk}}{} 
\fi

\definecolor{ChadDarkBlue}{rgb}{.1,0,.2}  
\definecolor{ChadBlue}{rgb}{.1,.1,.5}  
\definecolor{ChadRoyal}{rgb}{.2,.2,.8}  
\definecolor{ChadGreen}{rgb}{0,.6,0}    
\definecolor{ChadRed}{rgb}{.5,0,.5}  

\def\oprava#1{{#1}} 
\def\zmena#1{{#1}} 
\def\autori#1{{#1}} 

\usepackage{amssymb}
\usepackage{amsmath}
\usepackage{amsthm}
\usepackage{epic,eepic}
\usepackage[mathscr]{euscript}
\usepackage{enumerate}

\usepackage{graphicx}

\usepackage{lineno}
\usepackage{comment}

\def\smallskip{\vskip\smallskipamount}
\def\medskip{\vskip\medskipamount}
\def\bigskip{\vskip\bigskipamount}

\newtheorem{theorem}{Theorem}[section]
\newtheorem{lemma}[theorem]{Lemma}
\newtheorem{statement}[theorem]{Statement}
\newtheorem{proposition}[theorem]{Proposition}
\newtheorem{corollary}[theorem]{Corollary}

\newtheorem{example}[theorem]{Example}
\newtheorem{definition}[theorem]{Definition}
\newtheorem{observation}[theorem]{Observation}
\newtheorem{open}[theorem]{Open problem}

\def\operatorname#1{{\mathop{\rm #1}}}
\newcommand{\ord}{\operatorname{ord}}

\def\comp{\leftrightarrow}

\newcounter{ok}
{\end{list}}

\newcounter{aok}
{\end{list}}

\def\go#1;#2;#3 {\vbox to0pt{\kern-#3\hbox{\kern#2 #1}\vss}\nointerlineskip}

\newcommand{\Mea}{{\text{\rm{}M}}}
\newcommand{\HMea}{{\text{\rm{}HM}}}
\newcommand{\UMea}{{\text{\rm{}UM}}}

\newcommand{\C}{\text{\rm{}C}}
\newcommand{\Sh}{\text{\rm{}S}}

\newcommand{\itrm}[1]{\item[{\rm {#1}}]}
\begin{document}

\title{Homogeneous orthocomplete effect algebras are covered by MV-algebras}

\author[jp]{Josef~Niederle}
\ead{niederle@math.muni.cz}

\author[jp]{Jan~Paseka\corref{cor1}} 
\ead{paseka@math.muni.cz}
\address[jp]{Department of Mathematics and Statistics,
			Faculty of Science,
			Masaryk University,
			{Kotl\'a\v r{}sk\' a\ 2},
			CZ-611~37~Brno, Czech Republic}        	

\cortext[cor1]{Corresponding author}

\date{Received: date / Accepted: date}


\begin{abstract}
The aim of our paper is twofold. First, 
we thoroughly study the set of meager elements $\Mea(E)$ 
and the set of hypermeager elements $\HMea(E)$ 
in the setting of homogeneous effect algebras $E$. 
Second, we study the  property (W+) and the maximality property
introduced by Tkadlec   as common 
generalizations of orthocomplete and lattice effect algebras. 
We show that every block of an Archimedean homogeneous effect algebra 
satisfying the property (W+) is lattice ordered. Hence 
such effect algebras can be covered by ranges of observables. 
As a corollary, this yields that every block of a homogeneous orthocomplete effect algebra is lattice ordered. 
Therefore finite homogeneous effect algebras  are covered by MV-algebras.
\end{abstract}


\maketitle

\zmena{\section*{Introduction}}

\label{intro}
The history of quantum structures started at the beginning of the 20th century.
Observable events constitute a Boolean algebra in a classical physical system.
Because event structures in quantum mechanics cannot be described by Boolean algebras, 
Birkhoff and von Neumann introduced orthomodular lattices
which were considered as the standard quantum logic. Later on, orthoalgebras  were introduced
as the generalizations of orthomodular posets, which were considered
as "sharp"  quantum logic.

In the nineties of the twentieth century, two
equivalent quantum structures, D-posets and effect algebras were extensively
studied, which were considered as "unsharp" generalizations of the
structures which arise in quantum mechanics, in particular, of orthomodular
lattices and MV-algebras.

Effect algebras are fundamental in investigations of fuzzy probability
theory. In the fuzzy probability frame, the elements of an effect algebra
represent fuzzy events which are used to construct fuzzy random variables.

In the  present  paper,  we continue the study of 
homogeneous  effect algebras started in \cite{jenca2001}.  
This class of effect algebras includes
orthoalgebras,  lattice ordered effect  algebras and  effect  algebras satisfying  the Riesz decomposition  property.

In \cite{jenca2001} it was proved that  every homogeneous 
effect  algebra is a union of its blocks, which are defined as maximal sub-effect  algebras satisfying  the Riesz decomposition property.  
In \cite{tkadlec} Tkadlec introduced the  property (W+) as a common 
generalization of orthocomplete and lattice effect algebras.

Rie\v{c}anov\'{a} in \cite{ZR56} proved one of the most important results in the theory 
of effect algebras that each lattice ordered effect algebra 
can be covered by MV-subalgebras which form
blocks. Dvure{\v{c}}enskij extended  in \cite{dvurec2} this result 
for effect algebras with the Riesz interpolation property and with the decomposition-meet property. 
Pulmannov\'{a} \cite{pulmblok} proved that  every homogeneous effect algebra  
$E$  such that every block $B$ of  $E$   satisfies  the decomposition-meet property can  be covered  by   MV-algebras.

The aim of our paper is to show that every block of an Archimedean homogeneous effect  algebra satisfying the property (W+) is lattice ordered. Hence Archimedean homogeneous effect algebras 
fulfilling the condition (W+) can  be  covered  by  ranges of observables.
As a corollary, this yields that every block of a homogeneous orthocomplete effect algebra is lattice ordered. Therefore finite homogeneous effect algebras  are covered by MV-algebras which form blocks.

As a by-product of our study we extend the results on sharp and meager elements 
of \cite{jenca} 
into the realm of  Archimedean homogeneous effect  algebras 
satisfying the property (W+). We also thoroughly study the set of meager elements $\Mea(E)$ 
and the set of hypermeager elements $\HMea(E)$ 
in the setting of homogeneous effect algebras $E$.

\medskip

\section{{Preliminaries} and basic facts}

\zmena{Effect algebras were introduced by Foulis and 
 Bennett (see \cite{FoBe}) 
for modelling unsharp measurements in a Hilbert space. In this case the 
set $\EuScript E(H)$ of effects is the set of all self-adjoint operators $A$ on 
a Hilbert space $H$ between the null operator $0$ and the identity 
operator $1$ and endowed with the partial operation $+$ defined 
iff  $A+B$ is in $\EuScript E(H)$, where $+$ is the usual operator sum. }

\zmena{In general form, an effect algebra is in fact a partial algebra 
with one partial binary operation and two unary operations satisfying 
the following axioms due to Foulis and 
 Bennett.}
 
 The basic reference for the present text is the classic
book by Dvure{\v{c}}enskij and Pulmannov\'{a} \cite{dvurec}, where the interested reader can
find  unexplained terms and notation concerning the subject. 

\begin{definition}\label{def:EA}{\autori{{\rm{}\cite{FoBe,ZR62}  }}
{\rm A partial algebra $(E;\oplus,0,1)$ is called an {\em effect algebra} if
$0$, $1$ are two distinct elements, 
called the {\em zero} and the {\em unit}  element,  and $\oplus$ is a partially
defined binary operation called the {\em orthosummation} 
on $E$ which satisfy the following
conditions for any $x,y,z\in E$:
\begin{description}
\item[\rm(Ei)\phantom{ii}] $x\oplus y=y\oplus x$ if $x\oplus y$ is defined,
\item[\rm(Eii)\phantom{i}] $(x\oplus y)\oplus z=x\oplus(y\oplus z)$  if one
side is defined,
\item[\rm(Eiii)] for every $x\in E$ there exists a unique $y\in
E$ such that $x\oplus y=1$ (we put $x'=y$),
\item[\rm(Eiv)\phantom{i}] if $1\oplus x$ is defined then $x=0$.
\end{description}
}%
{\rm{}$(E;\oplus,0,1)$ is  called an {\em orthoalgebra} if 
$x\oplus x$ exists implies that $x= 0$  (see \cite{grerut}).}}
\end{definition}

We often denote the effect algebra $(E;\oplus,0,1)$ briefly by
$E$. On every effect algebra $E$  a partial order
$\le$  and a partial binary operation $\ominus$ can be 
introduced as follows:
\begin{center}
$x\le y$ \mbox{ and } {\autori{$y\ominus x=z$}} \mbox{ iff }$x\oplus z$
\mbox{ is defined and }$x\oplus z=y$\,.
\end{center}

If $E$ with the defined partial order is a lattice (a complete
lattice) then $(E;\oplus,0,1)$ is called a {\em lattice effect
algebra} ({\em a complete lattice effect algebra}).

Mappings from one effect algebra to
another one that preserve units and orthosums are called 
{\em morphisms of effect algebras}, and bijective
morphisms of effect algebras having inverses that are 
morphisms of effect algebras are called 
{\em isomorphisms of effect algebras}.

\begin{definition}\label{subef}{\rm
Let $E$ be an  effect algebra.
Then $Q\subseteq E$ is called a {\em sub-effect algebra} of  $E$ if 
\begin{enumerate}
\item[(i)] $1\in Q$
\item[(ii)] if out of elements $x,y,z\in E$ with $x\oplus y=z$
two are in $Q$, then $x,y,z\in Q$.
\end{enumerate}
If $E$ is a lattice effect algebra and $Q$ is a sub-lattice and a sub-effect
algebra of $E$, then $Q$ is called a {\em sub-lattice effect algebra} of $E$.}
\end{definition}

Note that a sub-effect algebra $Q$ 
(sub-lattice effect algebra $Q$) of an  effect algebra $E$ 
(of a lattice effect algebra $E$) with inherited operation 
$\oplus$ is an  effect algebra (lattice effect algebra) 
in its own right.

For an element $x$ of an effect algebra $E$ we write
$\ord(x)=\infty$ if $nx=x\oplus x\oplus\dots\oplus x$ ($n$-times)
exists for every positive integer $n$ and we write $\ord(x)=n_x$
if $n_x$ is the greatest positive integer such that $n_xx$
exists in $E$.  An effect algebra $E$ is {\em Archimedean} if
$\ord(x)<\infty$ for all $x\in E$.

A minimal nonzero element of an effect algebra  $E$
is called an {\em atom}  and $E$ is
called {\em atomic} if under every nonzero element of 
$E$ there is an atom.

\begin{definition}
\rm
We say that a finite system $F=(x_k)_{k=1}^n$ of not necessarily
different elements of an effect algebra $E$ is
\zmena{\em orthogonal} if $x_1\oplus x_2\oplus \cdots\oplus
x_n$ (written $\bigoplus\limits_{k=1}^n x_k$ or $\bigoplus F$) exists
in $E$. Here we define $x_1\oplus x_2\oplus \cdots\oplus x_n=
(x_1\oplus x_2\oplus \cdots\oplus x_{n-1})\oplus x_n$ supposing
that $\bigoplus\limits_{k=1}^{n-1}x_k$ is defined and
$(\bigoplus\limits_{k=1}^{n-1}x_k)\oplus x_n$ exists. We also define 
$\bigoplus \emptyset=0$.
An arbitrary system
$G=(x_{\kappa})_{\kappa\in H}$ of not necessarily different
elements of $E$ is called \zmena{\em orthogonal} if $\bigoplus K$
exists for every finite $K\subseteq G$. We say that for a \zmena{orthogonal} 
system $G=(x_{\kappa})_{\kappa\in H}$ the
element $\bigoplus G$ exists iff
$\bigvee\{\bigoplus K
\mid
K\subseteq G$ is finite$\}$ exists in $E$ and then we put
$\bigoplus G=\bigvee\{\bigoplus K\mid K\subseteq G$ is
finite$\}$. We say that $\bigoplus G$ is the {\em orthogonal sum} 
of $G$ and $G$ is {\em orthosummable}. (Here we write $G_1\subseteq G$ iff there is
$H_1\subseteq H$ such that $G_1=(x_{\kappa})_{\kappa\in
H_1}$).
{
We denote $G^\oplus:=\{\bigoplus K\mid K\subseteq G$ is   
finite$\}$.}
\end{definition}

\begin{definition}
\rm
$E$ is called \emph{orthocomplete} if  every orthogonal 
system  is orthosummable.
\end{definition}

Every orthocomplete effect algebra is
Archimedean.

{\renewcommand{\labelenumi}{{\normalfont  (\roman{enumi})}}
\begin{definition}
\rm
An element $x$ of an effect algebra $E$ is called  
\begin{enumerate}
\item  \emph{sharp} if $x\wedge x'=0$. The set 
$\Sh(E)=\{x\in E \mid x\wedge x'=0\}$ is called a \emph{set of all sharp elements} 
of $E$ (see \cite{gudder1}).
\item  \emph{principal}, if $y\oplus z\leq x$ for every $y, z\in E$ 
such that $y, z \leq x$ and $y\oplus z$ exists.
\item  \emph{central}, if $x$ and $x'$ are principal and,
for every $y \in E$ there are $y_1, y_2\in E$ such that 
$y_1\leq x, y_2\leq x'$, and $y=y_1\oplus y_2$  (see \cite{GrFoPu}). 
The \emph{center} 
$\C(E)$ of $E$ is the set of all central elements of $E$.
\end{enumerate}
\end{definition}

If $x\in E$ is a principal element, then $x$ is sharp and the interval 
$[0, x]$ is an effect algebra with the greatest element $x$ and the partial 
operation given by restriction of $\oplus$ to $[0, x]$. 

\begin{observation}
Clearly, $E$ is an orthoalgebra if and only if $\Sh(E)=E$.
\end{observation}

\begin{statement}{\rm \cite[Theorem 5.4]{GrFoPu}}
The center $\C(E)$ of an effect algebra $E$ is a sub-effect 
algebra of $E$ and forms a Boolean algebra. For every
central element $x$ of $E$, $y=(y\wedge x)\oplus (y\wedge x')$ for all 
$y\in E$.  If $x, y\in \C(E)$ are orthogonal, we have 
$x\vee y =  x\oplus y$  and $x\wedge y =  0$. 
\end{statement}

{\renewcommand{\labelenumi}{{\normalfont  (\roman{enumi})}}
\begin{statement}{\rm\cite[Lemma   3.1.]{jencapul}}\label{gejzapulm} Let $E$ be an effect algebra, 
$x, y\in E$ and $c, d\in \C(E)$. Then:
\begin{enumerate}
\item If $x\oplus y$ exists then  $c\wedge (x\oplus y)=(c\wedge x)\oplus (c\wedge y)$.
\item If $c\oplus d$ exists then  $x\wedge (c\oplus d)=(x\wedge c)\oplus (x\wedge d)$.  
\end{enumerate}
\end{statement}

\begin{definition}
\rm
A subset $M$ of an effect algebra $E$ is called 
\emph{compatible} (\emph{internally compatible}) if for every finite 
subset $M_F$ of $M$  there is a finite orthogonal family $(x_1, \dots, x_n)$ 
of elements in $E$ (in $M$) such that for every $m\in  M_F$ 
there is a set $A_F\subseteq \{ 1, \dots, n \}$ with 
$m =\bigoplus_{i\in A_F} x_i$. If $\{ x, y \}$ is a compatible set, 
we write $x\comp y$ (see \cite{jenca, Kop2}).
\end{definition}

Evidently, $x\comp y$ iff there are $p, q, r\in E$ such 
that $x=p\oplus q$, $y=q\oplus r$ and $p\oplus q\oplus r$ exists iff 
there are $c, d\in E$ such that $d\leq x\leq c$, $d\leq y\leq c$ and
$c\ominus x=y\ominus d$. 
Moreover, if $x\wedge y$ exists then $x\comp y$ iff 
$x\oplus (y\ominus (x\wedge y))$ exists.

\begin{definition}\label{rdp}{\rm
An effect algebra $E$  satisfies the \emph{Riesz decomposition property}
(or RDP) if, for all $u,v_1, v_2\in E$ such that
$u\leq v_1\oplus v_2$, there are $u_1, u_2$ such that 
$u_1\leq v_1, u_2\leq v_2$ and $u=u_1\oplus u_2$. 

An effect algebra $E$ is called \emph{homogeneous} if, 
for all $u,v_1,v_2\in E$ such that $u\leq v_1\oplus v_2\leq u'$, 
there are $u_1,u_2$ such that $u_1\leq v_1$, $u_2\leq v_2$ 
and $u=u_1\oplus u_2$ (see \cite{jenca2001}).

 An effect algebra $E$  satisfies the \emph{difference-meet property}  
(or DMP)  if, for all $x, y, z\in E$ such that $x\leq y$, $x\wedge z\in E$   
and  $y \wedge z\in E$,  then  $(y \ominus x) \wedge z  \in  E$ 
(see \cite{dvurec2}).}
\end{definition}

\begin{statement}\label{gejzablok}{\rm\cite[Proposition 2.3]{jenca2001}}
Let  $E$  be a homogeneous effect algebra.   Let  $u,v_1,\dots ,v_n\in  E$  
be  such  that  $v_1\oplus \dots \oplus v_n$   exists,  
$u\leq v_1\oplus \dots \oplus v_n\leq u'$.   Then there
are $u_1,\dots ,u_n$ such that, for all  
$1\leq i\leq  n$,  $u_i\leq v_i$ and  
$u  =  u_1\oplus \dots \oplus u_n$.
\end{statement}

\begin{statement}\label{gejzasum}{\rm\cite[Proposition 2]{jenca}}
\nopagebreak
\begin{enumerate}
\item[{\rm (i)}]    Every orthoalgebra is homogeneous.
\itrm{(ii)}    Every lattice effect algebra is homogeneous.
\itrm{(iii)}    An effect algebra  $E$ has the Riesz decomposition 
property if and only if  $E$ is homogeneous and compatible.

Let $E$ be a homogeneous effect algebra.
\itrm{(iv)}    A subset $B$ of $E$ is a maximal sub-effect 
algebra of $E$ with the Riesz decomposition property (such  
$B$ is called a {\em block} of  $E$) if and only if  $B$ 
is a maximal internally compatible subset of $E$ containing $1$.
\itrm{(v)} Every finite compatible subset of $E$ is a subset of 
some block. This implies that every homogeneous effect algebra 
is a union of its blocks.
\itrm{(vi)}    $\Sh(E)$ is a sub-effect algebra of $E$.
\itrm{(vii)}    For every block $B$, $\C(B) = \Sh(E)\cap B$.
\itrm{(viii)}    Let $x\in B$, where $B$ is a block of $E$. 
Then $\{ y \in E \mid  y \leq x\ \text{and}\ y \leq x'\}\subseteq  B$.
\end{enumerate}
\end{statement}

Hence the class of homogeneous effect algebras includes orthoalgebras, 
effect algebras satisfying the Riesz decomposition property 
and lattice effect algebras.

{
\begin{proposition}\label{modyjem} Let $E$ be a 
homogeneous effect algebra and $v\in E$.
The following conditions are equivalent.
\begin{enumerate}
\item $v\in \Sh(E)$;
\item $y\leq z$ whenever $w, y, z\in E$   such that $v=w\oplus z$,
$y\leq w'$ and $y\leq w$. 
\end{enumerate}
\end{proposition}
\begin{proof}
(i) $\implies$ (ii)
Evidently, there is a block, say $B$, such that it contains the following 
orthogonal system $\{y, w\ominus y, z, 1\ominus v\}$. Hence 
$B$ contains also $w$, $w'$ and $v\in \C(B)$. Since 
$1=w\oplus w'$ we obtain by 
Statement \ref{gejzapulm}, (ii) that 
$v=v\wedge_B w\oplus %
v\wedge_B w'=%
w\oplus %
v\wedge_B w'$. Subtracting $w$ we obtain 
$z=v\wedge_B w'$. Hence
 $y\leq w \leq v$ and $y\leq w'$ yields that $y\leq z$.
\newline
(ii) $\implies$  (i)
Let $y\in [0,v]\cap[0,v']$. Put $w=v$ and $z=0$. Immediately, $y\leq 0$.
\end{proof}
}

An important class of effect algebras 
was introduced by Gudder in \cite{gudder1} 
and \cite{gudder2}. Fundamental example is the 
standard Hilbert spaces effect algebra ${\cal E}({\mathcal H})$.

For an element $x$ of an effect algebra $E$ we denote
$$
\begin{array}{r c  l c l}
\widetilde{x}&=\bigvee_{E}\{s\in \Sh(E) \mid s\leq x\}&%
\phantom{xxxxx}&\text{if it exists and belongs to}\ \Sh(E)\phantom{.}\\
\wwidehat{x}&=\bigwedge_{E}\{s\in \Sh(E) \mid s\geq x\}&\phantom{xxxxx}&%
\text{if it exists and belongs to}\ \Sh(E).
\end{array}
$$

{
\begin{definition} \rm (\cite{gudder1},  \cite{gudder2}.) 
An effect algebra $(E, \oplus, 0,
1)$ is called {\em sharply dominating} if for every $x\in E$ there
exists $\wwidehat{x}$.
\end{definition}

Obviously, $\wwidehat{x}$ is the smallest sharp element  such that $x\leq
\wwidehat{x}$. That is $\wwidehat{x}\in \Sh(E)$ and if $y\in \Sh(E)$ satisfies
$x\leq y$ then $\wwidehat{x}\leq y$.

Recall that the following conditions are equivalent in any
effect algebra $E$.
\begin{itemize}
\item $E$ is sharply dominating; 
\item for every $x\in E$ there exists $\widetilde{x}\in \Sh(E)$ such
that $\widetilde{x}\leq x$ and if $u\in \Sh(E)$
satisfies $u\leq x$ then $u\leq \widetilde{x}$;
\item for every $x\in E$ there exist 
a smallest sharp element $\wwidehat{x}$ over $x$ and a greatest sharp 
element $\widetilde{x}$ below $x$.
\end{itemize}
}

As proved in \cite{cattaneo}, 
$\Sh(E)$ is always a sub-effect algebra in 
a sharply dominating  effect algebra $E$.

\begin{statement}\label{gejza}{\rm{}\cite[Proposition 15]{jenca}}
Let $E$ be a sharply dominating  effect algebra. 
Then every $x \in E$ has a
unique decomposition $x = x_S \oplus x_M$, where $x_S\in\Sh(E)$ and $x_M \in \Mea(E)$, 
namely $x=\widetilde{x}\oplus {(x\ominus \widetilde{x})}$.
\end{statement}

{
\begin{lemma}\label{hatrozdilu}
Let $E$ be a sharply dominating effect algebra and let $x\in E$.
\begin{equation*}
\wwidehat{x\ominus\widetilde{x}}=
\wwidehat{\wwidehat{x}\ominus x}=
\wwidehat{x}\ominus \widetilde{x} .
\end{equation*}
\end{lemma}

\begin{proof}
Clearly $x\ominus\widetilde{x} \leq \wwidehat{x}\ominus \widetilde{x} \in
\Sh(E)$ and
$\wwidehat{x}\ominus x \leq \wwidehat{x}\ominus \widetilde{x} \in \Sh(E)$.
Therefore $\wwidehat{x\ominus\widetilde{x}} \leq \wwidehat{x}\ominus
\widetilde{x}$ and
$\wwidehat{\wwidehat{x}\ominus x} \leq \wwidehat{x}\ominus \widetilde{x}$.
Now, by adding $\widetilde{x}$, we obtain
$$x=\widetilde{x}\oplus (x\ominus\widetilde{x}) \leq
\widetilde{x}\oplus\wwidehat{x\ominus\widetilde{x}} \leq \wwidehat{x}$$
which yields
$\widetilde{x}\oplus\wwidehat{x\ominus\widetilde{x}}= \wwidehat{x}$,
and
similarly
$$\widetilde{x}=\wwidehat{x}\ominus (\wwidehat{x}\ominus
\widetilde{x}) \leq \wwidehat{x}\ominus \wwidehat{\wwidehat{x}\ominus x}
\leq \wwidehat{x}\ominus (\wwidehat{x}\ominus x)=x$$
which yields $\widetilde{x}= \wwidehat{x}\ominus
\wwidehat{\wwidehat{x}\ominus x} $.
\end{proof}

\begin{lemma}\label{suplem}
Let $E$ be a sharply dominating effect algebra and let $x\in E$.
\begin{equation*}
\wwidehat{x}\ominus x=x'\ominus (\wwidehat{x})'=x'\ominus\widetilde{(x')}
\end{equation*}
and
\begin{equation*}
x \ominus \widetilde{x}=(\widetilde{x})'\ominus x'=\wwidehat{(x')}\ominus x'
\end{equation*}
\end{lemma}

\begin{proof}
Transparent.
\end{proof}
}

\section{Meager, hypermeager and ultrameager elements}

In what follows set (see \cite{jenca,wujunde})
$$\Mea(E)=\{x\in E \mid\ \text{if}\ v\in \Sh(E)\ \text{satisfies}\ v\leq x\ 
\text{then}\ v=0\}.$$

An element $x\in \Mea(E)$ is called {\em meager}. Moreover, $x\in \Mea(E)$ 
iff $\widetilde{x}=0$. Recall that $x\in \Mea(E)$, $y\in E$, $y\leq x$ implies 
$y\in \Mea(E)$ and $x\ominus y\in \Mea(E)$.

We also define

\begin{definition}
\rm
$$\HMea(E)=\{x\in E \mid\ \text{there is}\ y\in E\ \text{such that}\ x\leq y\ 
\text{and}\ x\leq y'\}$$
and
$$\UMea(E)=\{x\in E \mid\ \text{for every}\ y\in \Sh(E)\ \text{such that}\ x\leq
y\ 
\text{it holds}\ x\leq y\ominus x\}.$$
An element $x\in \HMea(E)$ is called {\em hypermeager},
an element $x\in \UMea(E)$ is called {\em ultrameager}.
\end{definition}

\begin{figure}[!t]
\begin{center}
\setlength{\unitlength}{3 mm}
\begin{picture}(36,30)
\multiput(6,9)(24,0){2}{\line(0,1){12}}
\multiput(6,9)(12,18){2}{\line(2,-1){12}}
\multiput(18,3)(-12,18){2}{\line(2,1){12}}
\multiput(6,21)(12,0){2}{\line(1,-1){12}}
\multiput(6,9)(12,0){2}{\line(1,1){12}}
\put(18,3){\line(0,1){24}}
\multiput(6,9)(0,12){2}{\blacken\circle{1}}
\multiput(18,3)(0,24){2}{\blacken\circle{1}}
\multiput(30,9)(0,12){2}{\blacken\circle{1}}
\multiput(18,9)(0,6){3}{\whiten\circle{1}}
\multiput(16.5,1.5)(0,12){2}{\framebox(3,3){ }}
\put(18,9){\oval(6,18)}
\put(3,9){\makebox(0,0){$a$}}
\put(3,21){\makebox(0,0){$a\oplus b$}}
\put(3,9){\makebox(0,0){$a$}}
\put(18,1){\makebox(0,0){$0$}}
\put(22,9){\makebox(0,0){$b$}}
\put(22,15){\makebox(0,0){$2b$}}
\put(22,21){\makebox(0,0){$3b=a\oplus c$}}
\put(18,29){\makebox(0,0){$1=4b=a\oplus b\oplus c$}}
\put(33,9){\makebox(0,0){$c$}}
\put(33,21){\makebox(0,0){$b\oplus c$}}
\put(24,3){\vector(-1,0){3}}
\put(24.5,3){\makebox(0,0)[l]{$\Mea(E)=\HMea(E)$}}
\put(28.5,24){\framebox(3,3){ }}
\put(33,24){\makebox(0,0)[bl]{$\UMea(E)$}}
\end{picture}
\caption{Example \ref{exa_1}}\label{fig_1}
\end{center}
\end{figure}

\begin{example}\label{exa_1}
{\rm{}In the  non-homogeneous non-sharply dominating effect algebra
depictured in Figure \ref{fig_1},
$\Mea(E)=\HMea(E)\ne \UMea(E)$.
Sharp elements are denoted in black. One can easily check that 
$E$ is a sub-effect algebra of the MV-effect algebra $[0, 1]\times [0, 1]$ 
such that $a\mapsto (\frac{3}{4},0), b\mapsto (\frac{1}{4}, 
\frac{1}{4}), c\mapsto (0, \frac{3}{4})$. Moreover, since 
$a\oplus c\not\in \Sh(E)$ we obtain that $\Sh(E)$ is not a 
sub-effect algebra of $E$.}
\end{example}

\begin{example}\label{exa_2}{\rm{}In the  non-homogeneous non-sharply dominating effect algebra
depictured in Figure \ref{fig_2},
$\Mea(E)\ne \HMea(E)\ne \UMea(E)$.

\begin{figure}[!t]
\begin{center}
\setlength{\unitlength}{3 mm}
\begin{picture}(36,42)
\multiput(6,9)(24,0){2}{\line(0,1){24}}
\multiput(6,9)(12,30){2}{\line(2,-1){12}}
\multiput(18,3)(-12,30){2}{\line(2,1){12}}
\multiput(6,33)(12,0){2}{\line(1,-2){12}}
\multiput(6,9)(12,0){2}{\line(1,2){12}}
\put(18,3){\line(0,1){36}}
\multiput(6,9)(0,24){2}{\blacken\circle{1}}
\multiput(18,3)(0,36){2}{\blacken\circle{1}}
\multiput(30,9)(0,24){2}{\blacken\circle{1}}
\multiput(18,9)(0,6){5}{\whiten\circle{1}}
\multiput(16.5,1.5)(0,12){2}{\framebox(3,3){ }}
\put(16.5,19.5){\framebox(3,3){ }}
\put(18,12){\oval(4.5,23)}
\put(18,15){\oval(6,30)}
\put(3,9){\makebox(0,0){$a$}}
\put(3,33){\makebox(0,0){$a\oplus b$}}
\put(3,9){\makebox(0,0){$a$}}
\put(18,1){\makebox(0,0){$0$}}
\put(22,9){\makebox(0,0){$b$}}
\put(22,15){\makebox(0,0){$2b$}}
\put(22,21){\makebox(0,0){$3b$}}
\put(22,27){\makebox(0,0){$4b$}}
\put(22,33){\makebox(0,0){$5b=a\oplus c$}}
\put(18,41){\makebox(0,0){$1=6b=a\oplus b\oplus c$}}
\put(33,9){\makebox(0,0){$c$}}
\put(33,33){\makebox(0,0){$b\oplus c$}}
\put(24,1.5){\vector(-1,0){3}}
\put(24,3){\vector(-1,0){3.75}}
\put(24.5,1.5){\makebox(0,0)[l]{$\Mea(E)$}}
\put(24.5,3){\makebox(0,0)[l]{$\HMea(E)$}}
\put(28.5,36){\framebox(3,3){ }}
\put(33,36){\makebox(0,0)[bl]{$\UMea(E)$}}
\end{picture}
\caption{Example \ref{exa_2}}\label{fig_2}
\end{center}
\end{figure}

Sharp elements are denoted in black. One can easily check that 
$E$ is a sub-effect algebra of the MV-effect algebra $[0, 1]\times [0, 1]$ 
such that $a\mapsto (\frac{5}{6},0), b\mapsto (\frac{1}{6}, 
\frac{1}{6}), c\mapsto (0, \frac{5}{6})$. Moreover, since 
$a\oplus c\not\in \Sh(E)$ we obtain that $\Sh(E)$ is not a 
sub-effect algebra of $E$.}
\end{example}

\begin{lemma}\label{semea}Let $E$ be an effect algebra. Then 
$\HMea(E)\subseteq \Mea(E)$. Moreover, for all $x\in E$, 
$x\in \HMea(E)$ iff $x\oplus x$ exists and, 
for all $y\in \Mea(E)$, $y\not =0$  there is $h\in \HMea(E)$, $h\not =0$ 
such that $h\leq y$.
\end{lemma}
\begin{proof} Let $x\in \HMea(E)$. Then {there is}\ $y\in E$\ {such that}\ 
$x\leq y$\ {and}\ $x\leq y'$. Therefore also $x\leq y\leq x'$, i.e. $x\oplus x$ exists. 
Let $v\in \Sh(E)$, $v\leq x\leq y$. Then 
$v\leq x\leq y'\leq v'$. Hence $v=v\wedge v'=0$, i.e., $x\in \Mea(E)$. 

Now, let $x\in E$ such that $x\oplus x$ exists. Then $x\leq x'$ and evidently $x\leq x$. 
Hence $x\in \HMea(E)$.

Assume that $y\in \Mea(E)$, $y\not =0$. Since $y$  is meager there is a non-zero element 
$h$ such that $h\leq y$ and $h\leq y'$ (otherwise we would have $y\in \Sh(E)$, 
a contradiction). This invokes that $h$ is hypermeager.
\end{proof}

\begin{lemma}\label{usubhmea}
Every ultrameager element is hypermeager.
\end{lemma}

\begin{proof}
Let $x$ be an ultrameager element of an effect algebra $E$.
Because $1\in \Sh(E)$, $x\leq 1\ominus x=x^\prime$, and by Lemma \ref{semea}
$x$ is hypermeager.
\end{proof}

\begin{lemma}\label{umsdea}
Let $E$ be a sharply dominating effect algebra and let $y\in E$. Then 
$y$ is ultrameager if and only if $y\leq \wwidehat{y}\ominus y$.
\end{lemma}

\begin{proof}
For every $s\in \Sh(E)$ for which $y\leq s$, it holds
$\wwidehat{y}\leq s$ and $\wwidehat{y}\in \Sh(E)$.
\end{proof}

\begin{lemma}\label{ujehmea}
In every homogeneous effect algebra $E$, $\UMea(E)=\HMea(E)$.
\end{lemma}

\begin{proof}
Let $E$ be a homogeneous effect algebra. By Lemma \ref{usubhmea},
$\UMea(E)\subseteq\HMea(E)$. Let conversely $x\in \HMea(E)$ and
$y\in \Sh(E)$ such that $x\leq y$.
There exists a block $B$ for which $x,y,x\oplus x,y\ominus x\in B$.
By Statement \ref{gejzasum} (vii), $y$ is central in $B$.
Therefore $x\oplus x\leq y$ and consequently $x\leq y\ominus x$.
\end{proof}

\begin{lemma}
In every sharply dominating homogeneous effect algebra $E$,
\begin{multline*}
\UMea(E)=\bigcup_{y\in \UMea(E)}{\downarrow}y\cap
{\downarrow}(\wwidehat{y}\ominus y)
=
\bigcup_{y\in \HMea(E)}{\downarrow}y\cap
{\downarrow}(\wwidehat{y}\ominus y)
=\\
\bigcup_{y\in \Mea(E)}{\downarrow}y\cap
{\downarrow}(\wwidehat{y}\ominus y)=
\bigcup_{y\in E}{\downarrow}y\cap
{\downarrow}(\wwidehat{y}\ominus y)=
\HMea(E).
\end{multline*}
\end{lemma}

\begin{proof}
By Lemma \ref{ujehmea}, $\UMea(E)=\HMea(E)$.
Lemma \ref{umsdea} yields
$$\UMea(E) \subseteq 
\bigcup_{y\in \UMea(E)}{\downarrow}y\subseteq
\bigcup_{y\in \UMea(E)}{\downarrow}y\cap
{\downarrow}(\wwidehat{y}\ominus y),$$
which implies
\begin{multline*}
\HMea(E)=\UMea(E)\subseteq\bigcup_{y\in \UMea(E)}{\downarrow}y\cap
{\downarrow}(\wwidehat{y}\ominus y)
\subseteq
\bigcup_{y\in \HMea(E)}{\downarrow}y\cap
{\downarrow}(\wwidehat{y}\ominus y)
\subseteq\\
\bigcup_{y\in \Mea(E)}{\downarrow}y\cap
{\downarrow}(\wwidehat{y}\ominus y)
\subseteq
\bigcup_{y\in E}{\downarrow}y\cap
{\downarrow}(\wwidehat{y}\ominus y)\subseteq \HMea(E).
\end{multline*}
\end{proof}

{
\begin{lemma}\label{pomocne}
In any 
homogeneous effect algebra $E$,
$${y}\wedge_B {z}=0 \iff
{y}\wedge {z}=0$$

holds for any block $B$ and ${y},{z}\in \C(B)$.
\end{lemma}

\begin{proof}
$\implies$
Since ${y}\wedge_B {z}=0$, it holds
${y}\leq {z}'$. Therefore
$w \in [0,{y}] \cap [0,{z}]$ implies
$w \in [0,{z}'] \cap [0,{z}]=\{0\}$.
\newline
$\Longleftarrow$
Trivial.
\end{proof}

\begin{lemma}\label{sober}
The following conditions are equivalent in any sharply dominating
homogeneous effect algebra $E$.
\begin{enumerate}
\item for any block $B$ and $v,y,z \in B$, it holds $\wwidehat{v}\in B$
and $y\wedge_B z=0 \implies \wwidehat{y}\wedge_B
\wwidehat{z}=0$;
\item for any block $B$ and $v,y,z \in B$, it holds $\wwidehat{v}\in B$,
and $y\wedge_B z=0 \implies \wwidehat{y}\wedge_B
\wwidehat{z}=0$
if furthermore $y,z \in B \cap \Mea(E)$;
\item for any block $B$ and $v,y,z \in B$, it holds $\wwidehat{v}\in B$
and $y\wedge_B z=0 \implies \wwidehat{y}\wedge
\wwidehat{z}=0$;
\item for any block $B$ and $v,y,z \in B$, it holds $\wwidehat{v}\in B$,
and $y\wedge_B z=0 \implies \wwidehat{y}\wedge
\wwidehat{z}=0$
if furthermore $y,z \in B \cap \Mea(E)$.
\end{enumerate}
\end{lemma} 

\begin{proof}
Clearly (i)$\iff$(iii) and (ii)$\iff$(iv) in virtue of Lemma \ref{pomocne}.
Furthermore, (i)$\implies$(ii) and (iii)$\implies$(iv).
\newline
(ii)$\implies$(i)
Let $B$ be a block in $E$ and $y,z \in B$. By assumption,
$\wwidehat{y},\wwidehat{z}\in B$. Therefore
$\wwidehat{y}\ominus y,\wwidehat{z}\ominus z\in B\cap \Mea(E)$.
By assumption, from Lemma \ref{hatrozdilu} and  the fact that 
 $({{y}\ominus \widetilde{y}})\wedge_B ({{z}\ominus \widetilde{z}})=0$ we get
$$(\wwidehat{y}\ominus \widetilde{y})\wedge_B
(\wwidehat{z}\ominus \widetilde{z})=
\wwidehat{{y}\ominus \widetilde{y}}\wedge_B
\wwidehat{{z}\ominus \widetilde{z}}=0.$$
Further, $0\ne w \in [0,\widetilde{y}] \cap [0,\wwidehat{z}]\cap B$ implies
$0\ne \wwidehat{w} \in [0,\widetilde{y}] \cap [0,\wwidehat{z}]$.
Clearly $\wwidehat{w} \wedge_B (z
\oplus (\wwidehat{z}\ominus z)) = \wwidehat{w} \ne 0$.
There exist $w_1 \leq z, w_2 \leq \wwidehat{z}\ominus z$ for which
$w_1\oplus w_2 = \wwidehat{w}\leq y$.  
Because $\wwidehat{w} \wedge_B z=0$, $w_1=0$ and
$\wwidehat{w}=w_2\leq \wwidehat{z}\ominus z\in \Mea(E)$.
Therefore $\wwidehat{w}=0$, a contradiction.

This yields $\wwidehat{y}\wedge_B\wwidehat{z}= %
(\wwidehat{y}\ominus \widetilde{y})\wedge_B\wwidehat{z}%
\oplus (\widetilde{y} \wedge_B \wwidehat{z})= %
(\wwidehat{y}\ominus \widetilde{y}) \wedge_B \wwidehat{z}=%
(\wwidehat{{y}\ominus \widetilde{y}})\wedge_B \wwidehat{z}$.

Applying the above considerations once more we obtain that 
$\wwidehat{y}\wedge_B\wwidehat{z}= %
(\wwidehat{{y}\ominus \widetilde{y}})\wedge_B \wwidehat{z}=%
(\wwidehat{{y}\ominus \widetilde{y}})\wedge_B (\wwidehat{z}\ominus \widetilde{z})=0$.
\end{proof}

\begin{definition} \rm
A sharply dominating homogeneous effect algebra is {\em sober} if it
satisfies the equivalent conditions in Lemma \ref{sober}.
\end{definition}

\begin{lemma}\label{soucethat}
Let $E$ be a homogeneous effect algebra, and $y\in E$ and $w\in \Sh(E)$ for which
$y\leq w$ and $ky$ exists. It holds $ky\leq w$.
\end{lemma}

\begin{proof}
The elements $y'$, $w$, $y,2y,\dots,ky$ belong to one block $B$.
For $k=1$ the statement holds. Suppose $2\leq k$ and the statement holds for
$k-1$. 
By Statement \ref{gejzapulm}, $w\wedge_B ky=%
(w\wedge_B (k-1)y)\oplus y=(k-1)y\oplus y=ky$.
Therefore $ky\leq w$.
\end{proof}

\begin{lemma}\label{hatato}
Let $E$  be an Archimedean homogeneous effect algebra.
For any $a\in E\smallsetminus\{0\}$ for which $a\wedge (n_a a)'=0$,
it holds $\wwidehat{a}$ exists and $\wwidehat{a}=n_a a$.
\end{lemma}

\begin{proof}
Let $a\in E$ such that $a\wedge (n_a a)'=0$. Clearly $a\nleq (n_a a)'$.
Suppose that there exists an element
$b\in E$, $b\leq n_a a$ and $b \leq (n_a a)'$. 
By Statement \ref{gejzablok} we have that there are $b_1, \dots, b_{n_a}$ 
such that
$b=b_1\oplus \dots \oplus b_{n_a}$ and $b_i\leq a$ for all $1\leq i\leq n_a$. 
Hence $b_i\leq a \wedge (n_a a)'=0$  for all $1\leq i\leq n$, i.e. $b=0$. Therefore 
$n_a a\in \Sh(E)$ and by Lemma \ref{soucethat}, the statement follows.
\end{proof}

Let us recall the following statement.

{\renewcommand{\labelenumi}{{\normalfont  (\roman{enumi})}}
\begin{statement}{\rm\cite[Theorem 2.10]{niepa}}\label{popismeager} Let $E$ be 
an atomic Archime\-dean  lattice
effect algebra and let $x\in \Mea(E)$. Let us denote 
$A_x=\{a \mid \ a\ \text{an atom of}\ E,$ $ a\leq x\}$ and, for any $a\in A_x$, 
we shall put $k^{x}_a=\text{max}\{ k\in {\mathbb N}\mid ka \leq x\}$. Then
\begin{enumerate}
\settowidth{\leftmargin}{(iiiii)}
\settowidth{\labelwidth}{(iii)}
\settowidth{\itemindent}{(ii)}
\item For any $a\in A_x$ we have $k^{x}_a < n_a$.
\item The set $F_x=\{k^{x}_a a \mid \ a\in A_x\}$ is orthogonal and 
$$
x=\bigoplus\{k^{x}_a a \mid a\ \text{an atom of}\ E,\ a\leq x\}= \bigvee F_x.
$$
Moreover, for all $B\subseteq A_x$ and all natural numbers   $l_b< n_b, b\in B$ such that 
$x=\bigoplus \{l_b b \mid b\in B\}$ we have that $B=A_x$ and 
$l_a=k^{x}_a$ for all $a\in A_x$ i.e., $F_x$ is the unique set  
of multiples of atoms from $A_x$ such that its orthogonal sum is $x$. 
\item If  $\widehat{x}$ exists then 
$$
\begin{array}{r c l}
\widehat{x}=\widehat{\widehat{x}\ominus x}&=&\bigoplus\{n_a a \mid a\ \text{an atom of}\ E,\ a\leq x\}\\%
\phantom{\huge I}&=&\bigvee\{n_a a \mid\ a\in A_x\}
\end{array}
$$
and 
$$
\begin{array}{r c l}
\widehat{x}\ominus x&=&\bigoplus\{(n_a-k^{x}_a) a \mid\ a\in A_x\}\\%
\phantom{\huge I}&=&\bigvee\{(n_a-k^{x}_a) a \mid\ a\in A_x\}.
\end{array}
$$
\end{enumerate}
\end{statement}}

\begin{statement}\label{lem:2.1} {\rm \cite[Theorem 2.1]{ZR57}}\/
Let $E$ be a lattice effect algebra. 
Assume $b\in E$,
$A\subseteq E$ are such that $\bigvee A$ exists in $E$ and $b\comp a$
for all $a\in A$. Then
\begin{enumerate}
\item[{\rm(a)}]
$b\comp \bigvee A$.
\item[{\rm(b)}]
$\bigvee\{b\wedge a: a\in A\}$ exists in $E$ and
equals $b\wedge(\bigvee A)$.
\end{enumerate}
\end{statement}

\begin{proposition}
Every atomic Archimedean sharply dominating lattice effect algebra is
sober.
\end{proposition}

\begin{proof} Let us check the condition (ii) from Lemma \ref{sober}.
Let  $B$ be a block of $E$ and assume that $v,y,z \in B$. Then 
$v=\widetilde{v}\oplus x$, $x\in \Mea(E)$. If $x=0$ we are finished. 
Assume that $x\ne 0$. We shall use the same notation 
as in Statement \ref{popismeager}. Recall that 
$\widehat{x}\ominus x=\widehat{v}\ominus v\in \Mea(E)$. 
Moreover, let $a\in A_x$. Then $a\leq x \leq v$ and 
$a\leq \widehat{x}\ominus x=\widehat{v}\ominus v\leq v'$. Hence 
$a\in B$ by  Statement \ref{gejzasum}, (viii). This yields that 
$(n_a-k^{x}_a)a\in B$ for all $a\in A_x$. Since 
$\widehat{x}\ominus x=%
\bigoplus\{(n_a-k^{x}_a) a \mid a\ \text{an atom of}\ E,\ a\leq x\}$ 
we have by Statement \ref{lem:2.1} that 
$\widehat{v}\ominus v=\widehat{x}\ominus x\in B$. Hence also 
$\wwidehat{v}=(\widehat{v}\ominus v)\oplus v\in B$.

Assume now that $y\in \Mea(E)$ and $y\wedge z=0$. 
Let us put 
$A_y=\{a \mid \ a\ \text{an}$ $\text{atom of}\ E,$ $ a\leq y\}$. 
Evidently, $a\wedge z=0$, $n_a a\wedge z\in B$ and $z\leq a'$ 
for all $a\in A_y$. Therefore 
by  Statement \ref{gejzablok} 
$n_a a\wedge z\leq n_a a=a \oplus \dots \oplus a$ yields that 
$n_a a\wedge z=b_1\oplus \dots b_n$, $b_i\leq n_a a\wedge z\wedge a=0$  
for all $a\in A_y$. 
Then Statements \ref{popismeager}, (iii) and  \ref{lem:2.1}, (ii) yield that 
$\widehat{y}\wedge z=\bigvee\{n_a a \mid\ a\in A_y\}\wedge z=%
\bigvee\{n_a a \wedge z\mid\ a\in A_y\}=0$.

Assume now that $y, z\in \Mea(E)$ and $y\wedge z=0$. Applying the same considerations as above once more we get that 
$\widehat{y}\wedge \widehat{z}=0$.
\end{proof}

}

{
\section{Meager elements in orthocomplete homogeneous effect algebras}

By definition and preceding results, orthocomplete homogeneous effect
algebras are always homogeneous, Archimedean, sharply dominating and fulfill
the following condition (W+).
}

\begin{definition}{\rm \cite{tkadlec}}
\rm
An effect algebra $E$ {\em fulfills the condition (W+)} if for each orthogonal subset $A\subseteq E$
and each two upper bounds $u,v$ of $A^\oplus$ there exists an upper bound
$w$ of $A^\oplus$ below $u,v$. 

An effect algebra $E$ has the {\em maximality property} 
if $\{ u, v \}$ has a maximal
lower bound $w$ for every $u, v \in E$.
\end{definition}

It is easy to see that an effect algebra $E$ has the maximality
property if and only if $\{ u, v \}$ has a maximal
lower bound $w$, $w\geq t$ for every $u, v, t \in E$ such that 
$t$ is a lower bound of $\{ u, v \}$. As noted 
in \cite{tkadlec} $E$ has the maximality
property if and only if  $\{ u, v \}$ has a minimal
upper bound $w$ for every $u, v \in E$.

\begin{statement}{\rm \cite[Theorem 2.2]{tkadlec}} 
Lattice effect algebras and orthocomplete effect algebras fulfill both the condition (W+) and the maximality property.
\end{statement}

{
\begin{statement} \label{minimax} {\rm\cite[Theorem 3.1]{tkadlec}}
Let $E$ be an Archimedean effect algebra fulfilling the condition (W+),
and let $y,z\in E$.
Every lower bound of $y,z$ is below a maximal one and
every upper bound of $y,z$ is above a minimal one. Then 
$E$ has the maximality property.
\end{statement}

\begin{proposition}\label{hcmea}
Let $E$ be an Archimedean effect algebra fulfilling the condition (W+). 
Then every meager element of $E$ is the
orthosum of a system of hypermeager elements.
\end{proposition}

\begin{proof}
Let $y\in \Mea(E)$. Consider the set $\EuScript A$ of all orthogonal systems, precisely
multisets, $A$ of hypermeager elements for which $y$ is an upper bound of $A^\oplus$.
Since the multiset union of any chain in $\EuScript A$ belongs to $\EuScript
A$, there exists a maximal element $Z$ in $\EuScript A$. Since $E$ is Archimedean any 
element of $Z$ is contained in $Z$ only finitely many times. If $y$ is not the
supremum of $Z^\oplus$, there exists by the condition (W+) an upper bound $z$ of $Z^\oplus$
for which $z<y$. Since $y$ is meager, $y\ominus z\ne 0$ is meager too, and therefore
there exists a non-zero hypermeager  element $h$ such that $h\leq y\ominus z$.  
Obviously the multiset sum $Z \uplus \{h\}$
belongs to $\EuScript A$, which contradicts the assumption of maximality of $Z$.
\end{proof}

The following statement generalizes \cite[Theorem 13]{jenca}. 

\begin{corollary}\label{shcmea}
Let $E$ be an Archimedean sharply dominating effect algebra fulfilling the condition (W+). 
Then every element  $x\in E$ is the sum of $\widetilde{x}$ and of the
orthosum of a system of hypermeager elements.
\end{corollary}

\begin{lemma} \label{minima}
Let $E$ be an  effect algebra having the maximality property,
let $u,v\in E$, and let $a,b$ be two maximal lower bounds of $u,v$.
There exist elements $y,z$ for which $y \leq u$, $z \leq v$, $a,b$ are
maximal lower bounds of $y,z$ and $y,z$ are minimal upper bounds of $a,b$.
\end{lemma}

\begin{proof}
Straightforward. 
\end{proof}

\begin{lemma}[Shifting lemma] \label{shifting}
Let $E$ be an effect algebra  having the maximality property,
let $u,v\in E$, and let $a_1,b_1$ be two maximal lower bounds of $u,v$.
There exist elements $y,z$ and two maximal lower bounds $a,b$ of $y,z$
for which $y \leq u$, $z \leq v$, $a \leq a_1$, $b \leq b_1$, $a \wedge
b=0$, $a,b$ are
maximal lower bounds of $y,z$ and $y,z$ are minimal upper bounds of $a,b$.
Furthermore, $(y \ominus a) \wedge (z \ominus a) = 0$,
$(y \ominus b) \wedge (z \ominus b) = 0$,
$(y \ominus a) \wedge (y \ominus b) = 0$,
$(z \ominus a) \wedge (z \ominus b) = 0$.
\end{lemma}

\begin{proof} Let $c$ be a maximal lower bound of  $a_1,b_1$. Let 
us put $y_1=u\ominus c$, $z_1=v\ominus c$, $a=a_1\ominus c$ and 
$b=b_1\ominus c$. Evidently, $a\wedge b=0$, $y_1 \leq u$, $z_1 \leq v$, 
$a \leq a_1$, $b \leq b_1$ and $a,b$ are
maximal lower bounds of $y,z$. By Lemma \ref{minima} there exist elements $y,z$ 
for which $y \leq y_1$, $z \leq z_1$, $a,b$ are
maximal lower bounds of $y,z$ and $y,z$ are minimal upper bounds of $a,b$.
\end{proof}

The Shifting lemma provides the following {\em minimax structure}.

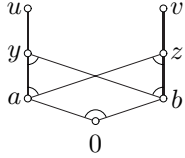
\newcommand{\minimax}{\begin{picture}(6,5)
\put(0,1){\line(0,1){4}}
\put(6,1){\line(0,1){4}}
\path(0,3)(6,1)(3,0)(0,1)(6,3)
\put(3,0){\whiten\circle{0.3}}
\put(0,1){\whiten\circle{0.3}}
\put(0,3){\whiten\circle{0.3}}
\put(0,5){\whiten\circle{0.3}}
\put(6,1){\whiten\circle{0.3}}
\put(6,3){\whiten\circle{0.3}}
\put(6,5){\whiten\circle{0.3}}

\put(3,0){\arc{1}{-2.77}{-0.37}}
\put(0,1){\arc{1}{-1.57}{-0.37}}
\put(0,3){\arc{1}{0.37}{1.57}}
\put(6,1){\arc{1}{-2.77}{-1.57}}
\put(6,3){\arc{1}{1.57}{2.77}}
\end{picture}
}

\newcommand{\maximin}{\begin{picture}(6,5)
\put(0,0){\line(0,1){4}}
\put(6,0){\line(0,1){4}}
\path(0,2)(6,4)(3,5)(0,4)(6,2)
\put(3,5){\whiten\circle{0.3}}
\put(0,0){\whiten\circle{0.3}}
\put(0,2){\whiten\circle{0.3}}
\put(0,4){\whiten\circle{0.3}}
\put(6,0){\whiten\circle{0.3}}
\put(6,2){\whiten\circle{0.3}}
\put(6,4){\whiten\circle{0.3}}

\put(3,5){\arc{1}{0.37}{2.77}}
\put(0,2){\arc{1}{-1.57}{-0.37}}
\put(0,4){\arc{1}{0.37}{1.57}}
\put(6,2){\arc{1}{-2.77}{-1.57}}
\put(6,4){\arc{1}{1.57}{2.77}}
\end{picture}
}

\begin{center}
\setlength{\unitlength}{3 mm}

\begin{picture}(6,5)
\put(0,0){\minimax}
\put(-0.6,1){\makebox(0,0){$a$}}
\put(-0.6,3){\makebox(0,0){$y$}}
\put(-0.6,5){\makebox(0,0){$u$}}
\put(6.6,1){\makebox(0,0){$b$}}
\put(6.6,3){\makebox(0,0){$z$}}
\put(6.6,5){\makebox(0,0){$v$}}
\put(3,-1){\makebox(0,0){$0$}}
\end{picture}
\vspace{1 mm}
\end{center}

\begin{proposition} \label{minimaxhyp}
Let $E$ be a homogeneous effect algebra 
having the maximality property.
Every two hypermeager elements $u,v$ possess $u \wedge v$.
\end{proposition}

\begin{proof}
Consider the minimax structure obtained by the Shifting lemma.
{\begin{center}
\setlength{\unitlength}{3 mm}
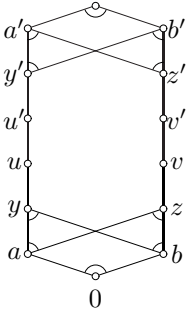
\begin{picture}(6,12)
\put(0,5){\line(0,1){2}}
\put(6,5){\line(0,1){2}}
\put(0,7){\maximin}
\put(0,0){\minimax}
\put(-0.6,1){\makebox(0,0){$a$}}
\put(-0.6,3){\makebox(0,0){$y$}}
\put(-0.6,5){\makebox(0,0){$u$}}
\put(6.6,1){\makebox(0,0){$b$}}
\put(6.6,3){\makebox(0,0){$z$}}
\put(6.6,5){\makebox(0,0){$v$}}
\put(-0.6,11){\makebox(0,0){$a'$}}
\put(-0.6,9){\makebox(0,0){$y'$}}
\put(-0.6,7){\makebox(0,0){$u'$}}
\put(6.6,11){\makebox(0,0){$b'$}}
\put(6.6,9){\makebox(0,0){$z'$}}
\put(6.6,7){\makebox(0,0){$v'$}}
\put(3,-1){\makebox(0,0){$0$}}
\end{picture}
\vspace{1 mm}
\end{center}}
Hence $a$ and $b$ are hypermeager and we have the following implications:

\noindent{}$a\leq(y\ominus b)\oplus b\leq a'
\implies
(\exists a_1\leq y\ominus b)(\exists a_2\leq b)\ a=a_1\oplus a_2
\overset{a\wedge b=0}{\implies}a_2=0, a=a_1\leq y\ominus b$ and\\
$a\leq(z\ominus b)\oplus b\leq a'
\implies
(\exists a_1\leq z\ominus b)(\exists a_2\leq b)\ a=a_1\oplus a_2
\overset{a\wedge b=0}{\implies}a_2=0, a=a_1\leq z\ominus b$.\\
Since $(y\ominus b)\wedge (z\ominus b)=0$, it follows $a=0$.
\end{proof}

\begin{proposition} \label{minimaxort}
Let $E$ be a homogeneous effect algebra 
having the maximality property.
For every orthogonal elements $u,v$, $u \wedge v$
and $u \vee_{[0,u \oplus v]} v$ exist and $[0,u \wedge 
v] \subseteq B$ for every block $B$ containing $u$ or $v$.
\end{proposition}

\begin{proof}
Consider the minimax structure obtained by the Shifting lemma.
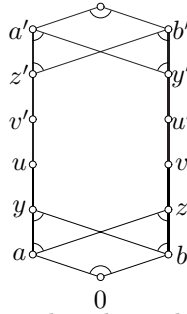
\begin{center}
\setlength{\unitlength}{3 mm}
\begin{picture}(6,12)
\put(0,5){\line(0,1){2}}
\put(6,5){\line(0,1){2}}
\put(0,7){\maximin}
\put(0,0){\minimax}
\put(-0.6,1){\makebox(0,0){$a$}}
\put(-0.6,3){\makebox(0,0){$y$}}
\put(-0.6,5){\makebox(0,0){$u$}}
\put(6.6,1){\makebox(0,0){$b$}}
\put(6.6,3){\makebox(0,0){$z$}}
\put(6.6,5){\makebox(0,0){$v$}}
\put(6.6,7){\makebox(0,0){$u'$}}
\put(-0.6,7){\makebox(0,0){$v'$}}
\put(-0.6,11){\makebox(0,0){$a'$}}
\put(6.6,9){\makebox(0,0){$y'$}}
\put(6.6,11){\makebox(0,0){$b'$}}
\put(-0.6,9){\makebox(0,0){$z'$}}
\put(3,-1){\makebox(0,0){$0$}}
\end{picture}
\vspace{1 mm}
\end{center}
Then $a$ and $b$ are hypermeager and we have the following implications:

\noindent{}$a\leq(y\ominus b)\oplus b\leq a'
\implies
(\exists a_1\leq y\ominus b)(\exists a_2\leq b)\ a=a_1\oplus a_2
\overset{a\wedge b=0}{\implies}a_2=0, a=a_1\leq y\ominus b$ and\\
$a\leq(z\ominus b)\oplus b\leq a'
\implies
(\exists a_1\leq z\ominus b)(\exists a_2\leq b)\ a=a_1\oplus a_2
\overset{a\wedge b=0}{\implies}a_2=0, a=a_1\leq z\ominus b$.\\ 

Since $(y\ominus b)\wedge (z\ominus b)=0$, it follows $a=0$.
Clearly $(u \oplus v) \ominus (u \wedge v)$ is the supremum of
$u,v$ in $[0,u \oplus v]$.

The remaining part of the Proposition follows by Statement 
\ref{gejzasum}, (viii).
\end{proof}

\begin{corollary} \label{minimaxsup}
Let $E$ be a homogeneous effect algebra 
having the maximality property.
For every element $u$, $u \wedge u'$ and $u \vee u'$ exist and $[0,u \wedge 
u'] \subseteq B$ for every block $B$ containing $u$.
\end{corollary}

\begin{corollary}\label{meetblock}
Let $E$ be a homogeneous effect algebra 
having the maximality property.
For any block $B$ and every elements $u,v\in B$
for which $u \wedge_B v=0$, $u \wedge v=0$.
\end{corollary}

\begin{proof}
Since $u \wedge_B v=0$, elements $u,v$ are orthogonal.
By Proposition~\ref{minimaxort}, $u \wedge v$ exists and $[0,u \wedge v]
\subseteq B$. Therefore $u \wedge v=0$.
\end{proof}

\begin{theorem}
Let $E$ be a homogeneous effect algebra having the maximality property.
Then every block $B$ in $E$ is a lattice, 
and therefore satisfies the {difference-meet property}.
\end{theorem}

\begin{proof}
Let $B$ be a block and $y,z \in B$. There exist elements $a,b,c \in B$ for
which $y=a \oplus b$, $z=a \oplus c$ and $a \oplus b \oplus c$ is defined.
Hence $b \wedge c$ exists and $b \wedge c \in B$ in virtue of
Proposition~\ref{minimaxort}.
Clearly, $a \oplus (b \wedge c)$ is a maximal lower bound of $y, z$.
Consequently, without loss of generality we may assume $b \wedge c=0$.
Suppose $v$ is a lower bound of $y,z$ in $B$. Hence  $v \leq a \oplus b$ and therefore
there exist $a_1 \leq a$ and $b_1 \leq b$ in $B$ for which $a_1 \oplus b_1=v$.
Now $v \ominus a_1=b_1 \leq b$.
Further $v \leq a \oplus c$. Therefore $v \ominus a_1 \leq (a \ominus a_1)
\oplus c$. There exist elements $a_2 \leq a \ominus a_1$ and $c_2 \leq c$ in
$B$ for which $v \ominus a_1=a_2 \oplus c_2$.
To sum up, $(v \ominus a_1) \ominus a_2=c_2 \leq c$ and
$(v \ominus a_1) \ominus a_2 \leq b \ominus a_2 \leq b$,
which together yields
$(v \ominus a_1) \ominus a_2=c_2 \leq b \wedge c=0$.
Consequently, $v=a_1 \oplus a_2 \leq a_1 \oplus (a \ominus a_1)=a$,
and $a$ is the infimum of $y,z$. This yields that $B$ is a lattice. 
\end{proof}

The preceding theorem immediately yields the following statements.

\begin{corollary} Let $E$ be a homogeneous effect algebra  having the maximality property.
Then  $E$  can be covered  by MV-algebras  which form
blocks.
\end{corollary}

\begin{corollary} Let $E$ be an Archimedean homogeneous effect algebra fulfilling the condition (W+).
Then  $E$  can be covered  by Archimedean  MV-algebras  which form
blocks.
\end{corollary}

Note that as in \cite{pulmblok} we obtain that Archimedean homogeneous effect algebras 
fulfilling the condition (W+) (in particular  orthocomplete 
homogeneous effect algebras) can  be  covered  by  ranges of observables. 

\begin{proposition}
Let $E$ be an orthocomplete homogeneous effect algebra. 
Then every block in $E$ is a lattice.
\end{proposition}

\begin{corollary}
Finite homogeneous effect algebras  are covered by MV-algebras.
\end{corollary}

\begin{corollary}
Let $E$ be a sharply dominating
homogeneous effect algebra  having the maximality property.
For any $y \in \Mea(E)$, $y \wedge (\wwidehat{y} \ominus y)$ exists and
$y \wedge (\wwidehat{y} \ominus y)=y \wedge y'$.
\end{corollary}

\begin{proof}
The meets exist in virtue of Proposition~\ref{minimaxort}.
Since $y \leq \wwidehat{y}$, there is a block $B$ for which $y,\wwidehat{y}
\in B$.
Now, $\wwidehat{y}=(y \oplus y') \wedge_B \wwidehat{y}=
(y \wedge_B \wwidehat{y}) \oplus (y' \wedge_B \wwidehat{y})=
y \oplus (y' \wedge_B \wwidehat{y})$, which implies
$\wwidehat{y} \ominus y=y' \wedge_B \wwidehat{y}$.
Consequently,
$y \wedge y'=(\wwidehat{y} \wedge y) \wedge y'=
(\wwidehat{y} \wedge_B y') \wedge_B y)=
(\wwidehat{y} \ominus y) \wedge_B y=
(\wwidehat{y} \ominus y) \wedge y$ because
$\wwidehat{y} \ominus y,y$ are orthogonal.
\end{proof}
}

\begin{figure}[!t]
\begin{center}
\quad\quad\quad\quad
\setlength{\unitlength}{1 cm}
\begin{picture}(10,7)
\path(5,0)(0,2)\path(5,0)(2,2)\path(5,0)(4,2)
\path(5,0)(6,2)\path(5,0)(8,2)\path(5,0)(10,2)
\path(5,7)(0,5)\path(5,7)(2,5)\path(5,7)(4,5)
\path(5,7)(6,5)\path(5,7)(8,5)\path(5,7)(10,5)
\path(0,2)(0,5)(2,2)(2,5)(0,2)(6,5)(8,2)(4,5)(2,2)(10,5)(4,2)(2,5)(6,2)(8,5)(8,2)(10,5)(10,2)(8,5)(0,2)
\multiput(0,2)(2,0){6}{\whiten\circle{0.3}}
\multiput(0,5)(2,0){6}{\whiten\circle{0.3}}
\multiput(5,0)(0,7){2}{\whiten\circle{0.3}}
{\scriptsize
\put(3.6,0){\makebox(0,0){${ (0, 0, 0)}$}}
\put(3.6,7){\makebox(0,0){${ (1, 1, 1)}$}}
\put(6,0){\makebox(0,0){$0$}}
\put(6,7){\makebox(0,0){$1$}}
\put(0,1.5){\makebox(0,0){$a\phantom{'}$}}
\put(2,1.5){\makebox(0,0){$b\phantom{'}$}}
\put(4,1.5){\makebox(0,0){$c\phantom{'}
$}}
\put(6,1.5){\makebox(0,0){$d\phantom{'}
$}}
\put(8,1.5){\makebox(0,0){$e\phantom{'}
$}}
\put(10,1.5){\makebox(0,0){$f\phantom{'}
$}}
\put(10,5.5){\makebox(0,0){$a'
$}}
\put(8,5.5){\makebox(0,0){$b'
$}}
\put(6,5.5){\makebox(0,0){$c'
$}}
\put(4,5.5){\makebox(0,0){$d'
$}}
\put(2,5.5){\makebox(0,0){$e'
$}}
\put(0,5.5){\makebox(0,0){$f'
$}}
}
{\scriptsize
\put(-0.25,2){\makebox(0,0)[r]{$ {
(\frac{17}{20},\frac{1}{20},\frac{1}{10})}$}}
\put(1.75,2){\makebox(0,0)[r]{$ { (\frac{1}{20},
\frac{17}{20},\frac{1}{10})}$}}
\put(3.75,2){\makebox(0,0)[r]{$ { (\frac{3}{20},
\frac{19}{20},0)}$}}
\put(5.75,2){\makebox(0,0)[r]{$ { (\frac{19}{20}, \frac{3}{20},
0)}$}}
\put(7.75,2){\makebox(0,0)[r]{$ { (0,
{0},\frac{9}{10})}$}}
\put(9.75,2){\makebox(0,0)[r]{$ { (\frac{1}{10},
\frac{1}{10},\frac{8}{10})}$}}
\put(9.75,5){\makebox(0,0)[r]{$ {
(\frac{3}{20},\frac{19}{20},\frac{9}{10})}$}}
\put(7.75,5){\makebox(0,0)[r]{$ { (\frac{19}{20},
\frac{3}{20},\frac{9}{10})}$}}
\put(5.75,5){\makebox(0,0)[r]{$ { (\frac{17}{20},
\frac{1}{20},1)}$}}
\put(3.75,5){\makebox(0,0)[r]{$ { (\frac{1}{20}, \frac{17}{20},
1)}$}}
\put(1.75,5){\makebox(0,0)[r]{$ { (1, {1},\frac{1}{10})}$}}
\put(-0.25,5){\makebox(0,0)[r]{$ { (\frac{9}{10},
\frac{9}{10},\frac{2}{10})}$}}
}

\end{picture}
\caption{Example \ref{exa_3}}\label{fig_3}
\end{center}
\end{figure}
\begin{example}\label{exa_3}{\rm{}In the  finite  orthoalgebra 
$E=\{0, a, b, c, d, e, f, a', b', $ $c', d', e',$ $f', 1\}$ 
depictured in Figure \ref{fig_3} which is such that 
$E=\Sh(E)$ (hence $E$ is homogeneous, Archimedean and orthocomplete),
finite joins in blocks do not coincide with finite joins in $E$.
One can easily check that 
$E$ is a sub-effect algebra of the MV-effect algebra 
$[0, 1]\times [0, 1]\times [0, 1]$ 
such that 
$$
\begin{array}{l l l}
0\phantom{'}\mapsto (0, 0, 0), & 1\phantom{'}\mapsto (1, 1, 1), &\\
a\phantom{'}\mapsto (\frac{17}{20},\frac{1}{20},\frac{1}{10}), &
b\phantom{'}\mapsto (\frac{1}{20}, \frac{17}{20},\frac{1}{10}),& 
c\phantom{'}\mapsto (\frac{3}{20}, \frac{19}{20},0),\phantom{\mbox{\huge I}}\\
d\phantom{'}\mapsto (\frac{19}{20}, \frac{3}{20}, 0),&
e\phantom{'}\mapsto (0, {0},\frac{9}{10}), &
f\phantom{'}\mapsto (\frac{1}{10}, \frac{1}{10},\frac{8}{10}),\phantom{\mbox{\huge I}}\\
a'\mapsto (\frac{3}{20},\frac{19}{20},\frac{9}{10}), &
b'\mapsto (\frac{19}{20}, \frac{3}{20},\frac{9}{10}),& 
c'\mapsto (\frac{17}{20}, \frac{1}{20},1),\phantom{\mbox{\huge I}}\\
d{'}\mapsto (\frac{1}{20}, \frac{17}{20}, 1),&
e{'}\mapsto (1, {1},\frac{1}{10}), &
f{'}\mapsto (\frac{9}{10}, \frac{9}{10},\frac{2}{10}).\phantom{\mbox{\huge I}}
\end{array}$$ 
Moreover $1=(1,1,1)=a\oplus b\oplus f=a\oplus c \oplus e=b\oplus d\oplus e$. 
This yields that $E$ has only the following blocks: 
$B_1=\{0, a, b, f, a\oplus b=f', a\oplus f=b', b\oplus f=a', 1\}$, 
$B_2=\{0, a, c, e, a\oplus c=e', a\oplus e=c', c\oplus e=a', 1\}$ and 
$B_3=\{0, b, d, e, b\oplus d=e', b\oplus e=d', d\oplus e=b', 1\}$ which are lattice ordered. 
Hence $a\not\comp d$, $b\not\comp c$, 
$c\not\comp d$, $c\not\comp f$,
$d\not\comp f$, $e\not\comp f$. In particular, 
$a$ and $b$ have two different minimal upper bounds, 
$a\oplus b$ and $a\oplus c$.
}
\end{example}

\begin{open}
One question still unanswered is whether,
 if  $A, B$   are two  blocks of $E$ with  
$x, y  \in   A \cap B$,   then 
$x\vee_A y=x\vee_B y$; here $E$ is 
an Archimedean sharply dominating
homogeneous effect algebra fulfilling the condition (W+). 
\end{open}

\section*{Acknowledgements}  J. Paseka gratefully acknowledges Financial Support 
of the  Ministry of Education of the Czech Republic
under the project MSM0021622409 and of Masaryk University under the grant 0964/2009.  
Both authors acknowledge the support by ESF Project CZ.1.07/2.3.00/20.0051
Algebraic methods in Quantum Logic of the Masaryk University.

\end{document}